\documentclass[11pt, reqno]{amsart}

\usepackage[T1]{fontenc}
\usepackage{amsmath}
\usepackage{amssymb}
\usepackage{amsthm}
\usepackage[left=3.5cm, right=3.5cm]{geometry}
\usepackage{hyperref}
\usepackage{fancyhdr}
\usepackage{enumitem}
\usepackage{comment}
\usepackage{nicefrac}
\usepackage{bm}
\usepackage{mathrsfs}
\usepackage{graphicx}
\usepackage[utf8]{inputenc}
\usepackage{cancel}
\usepackage{mathtools}

\newtheorem{thm}{Theorem}[section]
\newtheorem{cor}[thm]{Corollary}

\newtheorem{prop}[thm]{Proposition}
\newtheorem{question}[thm]{Question}

\theoremstyle{definition} 
\newtheorem{defi}[thm]{Definition}
\newtheorem{rmk}[thm]{Remark}

\numberwithin{equation}{section}

\theoremstyle{remark}
\newtheorem{claim}{\textsc{Claim}}
\newtheorem*{claim*}{\textsc{Claim}}

\newcounter{smallromans}

{\end{list}}

\AtBeginDocument{%
   \def\MR#1{}
}

\author[S.~G\l \k{a}b]{Szymon G\l \k{a}b}
\address{Institute of Mathematics, Lodz University of Technology, al. Politechniki 8, 93-590 Lodz, Poland}
\email{szymon.glab@p.lodz.pl}
\author[P.~Leonetti]{Paolo Leonetti}
\address{Department of Economics, Universit\`a degli Studi dell'Insubria, via Monte Generoso 71, 21100 Varese, Italy}
\email{leonetti.paolo@gmail.com}

\hypersetup{
    pdftitle={On the complexity of hypercyclic vectors},
    pdfauthor={},
    pdfmenubar=false,
    pdffitwindow=true,
    pdfstartview=FitH,
    colorlinks=true,
    linkcolor=blue,
    citecolor=green,
    urlcolor=cyan
}

\providecommand{\MR}[1]{}

\providecommand{\MR}{\relax\ifhmode\unskip\space\fi MR }

\keywords{Upper frequently hypercyclicity; 
upper and lower asymptotic density; 
analytic $P$-ideals; 
convergence in pointwise topology; 
weighted backward shifts.}

\subjclass[2010]{Primary: 37B20, 47A16. Secondary: 11B05, 37B99.}

\begin{document}
\title{On the complexity of upper frequently \\ hypercyclic vectors}

%


\begin{abstract} 
\noindent 
Given a continuous linear operator $T:X\to X$, where $X$ is a topological vector space, let $\mathrm{UFHC}(T)$ be the set of upper frequently hypercyclic vectors, that is, the set of vectors $x \in X$ such that $\{n \in \omega: T^nx \in U\}$ has positive upper asymptotic density for all nonempty open sets $U\subseteq X$. 
It is known that $\mathrm{UFHC}(T)$ is a $G_{\delta\sigma\delta}$-set which is either empty or contains a dense $G_{\delta}$-set. Using a purely topological proof, we improve it by showing that $\mathrm{UFHC}(T)$ is always a $G_{\delta\sigma}$-set. 

Bonilla and Grosse-Erdmann asked in [Rev. Mat. Complut. \textbf{31} (2018), 673--711] whether $\mathrm{UFHC}(T)$ is always a $G_{\delta}$-set. We answer such question in the negative, by showing that 
there exists a continuous linear operator $T$ for which $\mathrm{UFHC}(T)$ is not a $F_{\sigma\delta}$-set (hence not $G_\delta$). 
In addition, we study the [non-]equivalence between (the ideal versions of) upper frequently hypercyclicity in the product topology and upper frequently hypercyclicity in the norm topology.
\end{abstract}
\maketitle
\thispagestyle{empty}

\section{Introduction}\label{sec:intro}

Let $\mathsf{I}\subseteq \mathcal{P}(\omega)$ be an ideal, that is, a nonempty family of subsets of the nonnegative integers $\omega$ which is stable under taking subsets and finite unions. Unless otherwise stated, it is also assumed that $\{n\} \in \mathsf{I}$ for all $n \in \omega$ and that $\omega \notin \mathsf{I}$. 
Informally, an ideal represents the family of \textquotedblleft small sets.\textquotedblright\,  
Notable examples of ideals are the family $\mathrm{Fin}$ of finite subsets of $\omega$, the family of asymptotic density zero sets 
$$
\mathsf{Z}:=\{S\subseteq \omega: |S \cap [0,n]|=o(n) \text{ as }n\to \infty\},
$$
the family of logarithmic density zero sets 
$$
\mathsf{Z}_{\mathrm{log}}:=\left\{S\subseteq \omega: \sum\nolimits_{k \in S \cap (0,n]}1/k=o(\log(n)) \text{ as }n\to \infty\right\},
$$
the summable ideal 
$$
\textstyle 
\mathsf{I}_{1/n}:=\left\{S\subseteq \omega: \sum\nolimits_{n \in S}1/(n+1)<\infty\right\},
$$
and the complements of free ultrafilters on $\omega$. 
An ideal $\mathsf{I}$ is said to be a $P$\emph{-ideal} if it is $\sigma$-directed modulo finite sets, that is, for all sequences $(S_n) \in \mathsf{I}^\omega$ there exists $S \in \mathsf{I}$ such that $S_n\setminus S$ is finite for all $n \in \omega$. 
Identifying $\mathcal{P}(\omega)$ with the Cantor space $\{0,1\}^\omega$, it is possible to speak about $F_\sigma$-ideals, Borel ideals, analytic ideals, etc.; in particular, it is easy to see that $\mathsf{Fin}$ and $\mathsf{I}_{1/n}$ are $F_\sigma$ $P$-ideals (hence, also analytic), while it is known that both $\mathsf{Z}$ and $\mathsf{Z}_{\mathrm{log}}$ are analytic $P$-ideals which are not $F_\sigma$. In addition, $\mathsf{Z}\subseteq \mathsf{Z}_{\mathrm{log}}$. 
We refer the reader to 
\cite{MR1711328} for an excellent textbook on the theory of ideals on $\omega$. 

Given a sequence $\bm{x}=(x_n: n \in \omega)$ taking values in a topological space $X$ and an ideal $\mathsf{I}$ on $\omega$, we say that $\eta\in X$ is an $\mathsf{I}$\textbf{-cluster point} of $\bm{x}$ if 
$$
\{n \in \omega: x_n \in U\}\in \mathsf{I}^+
$$ 
for all neighborhood $U$ of $\eta$, where $\mathsf{I}^+:=\mathcal{P}(\omega)\setminus \mathsf{I}$. The family of $\mathsf{I}$-cluster points of $\bm{x}$ is denoted by $\Gamma_{\bm{x}}(\mathsf{I})$. It is known that $\Gamma_{\bm{x}}(\mathsf{I})$ is always closed, see \cite[Lemma 3.1]{MR3920799}. 
In the literature, $\mathsf{Z}$-cluster points are usually called \emph{statistical cluster points}, see e.g. \cite{MR1181163, MR1416085}.  

At this point, suppose that $X$ is a real topological vector space, and let $T: X\to X$ be a continuous linear operator. A vector $x \in X$ is said to be $\bm{\mathsf{I}}$\textbf{-hypercyclic} if every vector in $X$ is an $\mathsf{I}$-cluster point of the 
orbit of $x$ with respect to $T$.  
Hence, we denote by $\mathrm{HC}_T(\mathsf{I})$ the set of $\mathsf{I}$-hypercyclic vectors , namely, 
$$
\mathrm{HC}_T(\mathsf{I}):=\left\{x \in X: \Gamma_{\mathrm{orb}(x,T)}(\mathsf{I})=X\right\}, 
$$
where $\mathrm{orb}(x,T):=(T^nx: n \in \omega)$ and, by convention, $T^0x:=x$. 
We remark that $\mathrm{Fin}$-hypercyclic vectors and $\mathsf{Z}$-hypercyclic vectors are usually called \emph{hypercyclic} and \emph{upper frequently hypercyclic} vectors, respectively, and their sets are commonly denoted by 
$$
\mathrm{HC}(T):=\mathrm{HC}_T(\mathrm{Fin})
\,\,\,\,\,\,\,\,\text{ and }\,\,\,\,\,\,\,\,
\mathrm{UFHC}(T):=\mathrm{HC}_T(\mathsf{Z}),
$$
respectively; see e.g. \cite{MR3847081, MR4242546, MR4016510} and the excellent textbooks \cite{MR2533318, MR2919812}. 

The following result has been shown in \cite{Leo24}, cf. also \cite{MR3847081}: 
\begin{thm}\label{thm:startingthm}
Let $T: X\to X$ be a continuous linear operator, where $X$ is a 
second countable topological vector space, 
and let $\mathsf{I}$ be an ideal on $\omega$. 
Then the following hold\textup{:}
\begin{enumerate}[label={\rm (\roman{*})}]
\item \label{item:01} $\mathrm{HC}_T(\mathsf{I})$ is either empty or dense\textup{;}
\item \label{item:02} If $\mathsf{I}$ is a $F_\sigma$-ideal, then $\mathrm{HC}_T(\mathsf{I})$ is a $G_\delta$-set\textup{;} 
\item \label{item:03} If $\mathsf{I}$ is an analytic $P$-ideal, then $\mathrm{HC}_T(\mathsf{I})$ is a $G_{\delta\sigma\delta}$-set\textup{.} 
\end{enumerate}
%
%
%
\end{thm}
\begin{proof}
    See \cite[Proposition 3.2, Corollary 3.5, and Theorem 3.6(i)]{Leo24}. 
\end{proof}

Taking into account that $\mathrm{Fin}$ is a $F_\sigma$-ideal (since it is a countable family), it follows that the set $\mathrm{HC}(T)$ of ordinary hypercyclic vectors  
is either empty or a dense $G_\delta$-set, hence either empty or comeager. A similar claim can be shown for the ideal $\mathsf{Z}$ of asymptotic density zero sets (which is an analytic $P$-ideal, but not $F_\sigma$): the set $\mathrm{UFHC}(T)$ of upper frequently hypercyclic vectors 
is either empty or contains a dense $G_{\delta}$-set, hence either empty or comeager, see e.g. \cite[Theorem 3.7]{Leo24}. 


Our first result improves on Theorem \ref{thm:startingthm}\ref{item:03}: 
\begin{thm}\label{thm:topologicalanalyticP}
Let $T: X\to X$ be a continuous linear operator, where $X$ is a 
second countable topological vector space, 
and let $\mathsf{I}$ be an analytic $P$-ideal on $\omega$. Then $\mathrm{HC}_T(\mathsf{I})$ is a $G_{\delta\sigma}$-set. In particular, $\mathrm{UFHC}(T)$ is a $G_{\delta\sigma}$-set. 
\end{thm}

The following question has been asked by Bonilla and Grosse-Erdmann in \cite[p. 683]{MR3847081}: 
\begin{question}\label{question:Karl}
Is it true that the set $\mathrm{UFHC}(T)$ of upper frequently hypercyclic vectors is always a $G_\delta$-set? 
\end{question}

We answer Question \ref{question:Karl} in the negative.  
In fact, we are going to show that 
there exists a unilateral weighted backward shift on 
$\ell_p$ for which $\mathrm{UFHC}(T)$ is not a $F_{\sigma\delta}$-set (hence, not $G_\delta$), provided that $\ell_p$ is endowed with the product topology; see Theorem  \ref{thm:mainnegativekarl}. 
%
In addition, we study the equivalence (and non-equivalence) between $\mathsf{I}$-hypercyclicity in the product topology and $\mathsf{I}$-hypercyclicity in the norm topology for certain ideals $\mathsf{I}$, see Theorem \ref{thm:equivalenceweaklynorm} and Theorem \ref{thm:nonequivalence}. 
The proofs of all our results 
(including Theorem \ref{thm:topologicalanalyticP} above) 
will be given in Section \ref{sec:proof22}. 

\section{Preliminaries and Main results}

First, we need to fix some notation and recall some standard results. Recall that if $X\subseteq \mathbf{R}^\omega$ is a \emph{Banach sequence space} (that is, a Banach space with the property that the embedding $X\to \mathbf{R}^\omega$ is continuous, see e.g. \cite[Chapter 4]{MR2919812}) and if $\bm{w}=(w_n: n \in \omega)$ is a sequence of nonzero reals, called \emph{weight sequence}, then the (unilateral) \textbf{weighted backward shift} is a map $B_{\bm{w}}: X\to X$ defined by 
\begin{displaymath}
    B_{\bm{w}}(x_0,x_1,x_2,\ldots):=(w_1x_1,w_2x_2,w_3x_3,\ldots)  
\end{displaymath}
for all $x=(x_0,x_1,x_2,\ldots) \in X$. 
Note that the above definition requires that a weighted shift $B_{\bm{w}}$ maps $X$ into itself. 
In addition, the value $w_0$ in the weight sequence $\bm{w}$ is irrelevant. 
Unless otherwise noted, all weighted backward shifts are unilateral. 
In the case where $\bm{w}$ is the constant sequence $(1,1,\ldots)$, we simply write $B:=B_{\bm{w}}$. 

In the next result (which answers Question \ref{question:Karl}) 
the classical Banach sequence space $\ell_p$, with $p \in [1,\infty)$, is endowed with the relative topology of the product topology on $\mathbf{R}^\omega$, which will be denoted with $\tau^{p}$. 
Of course, every weighted backward shift $B_{\bm{w}}$ is $\tau^p$-$\tau^p$ continuous. We remark the product topology $\tau^p$ already appeared in the literature:  for instance, by a result of Grosse-Erdmann, every $\tau^p$-hypercyclic operator on $\mathbf{R}^\omega$ satisfies the Hypercyclicity Criterion, see \cite[Proposition 6.1]{MR2352487}; see also \cite{MR4072792, MR1763044}. 

\begin{thm}\label{thm:mainnegativekarl}
Let $\bm{w}$ be a bounded sequence of reals with $w_n\ge 1$ for all $n \in \omega$ and consider the weighted backward shift $B_{\bm{w}}$ on $\ell_p$, with $p \in [1,\infty)$, endowed with the product topology. Suppose also that the sequence $\bm{w}$ satisfies  \begin{equation}\label{eq:conditionufhc}
     \sum_{n \in \omega}\frac{1}{(w_0\cdots w_n)^p}<\infty.
 \end{equation}
 
 Then $\mathrm{UFHC}(B_{\bm{w}})$ is not a $F_{\sigma\delta}$-set. 

 In particular, $\mathrm{UFHC}(\lambda B)$ is not a $G_\delta$-set for all $\lambda>1$. 
\end{thm}

We underline again that $\mathrm{UFHC}(B_{\bm{w}})$ in the statement above is the set of upper frequently hypercyclic vectors $x \in \ell_p$ with respect to topology $\tau^p$. In Theorem \ref{thm:nonequivalence} below we will show, in particular, that the latter set does \emph{not} necessarily coincide with the the set upper frequently hypercyclic vectors with respect to the norm topology.  

In the following, 
given a Banach sequence space $X\subseteq \mathbf{R}^\omega$, we write $\tau^n$ for the norm topology on $X$, and  $\tau^w$ for the weak topology on $X$. 
We recall that, as a consequence of the closed graph theorem, every weighted backward shift $B_{\bm{w}}: X\to X$ is norm-to-norm continuous, see \cite[Proposition 4.1]{MR2919812}. 
On this line, recall also that a bounded linear operator between two Banach spaces is norm-to-norm continuous if and only if it is weak-to-weak continuous, see \cite[p. 166]{MR1070713}. 
Since $\tau^p$ stands for the (subspace topology on $X$ inherited from the) pointwise topology, we have 
$$
\tau^p \subseteq \tau^w\subseteq \tau^n.
$$ 
It is well known that, if $X$ is infinite dimensional, then $\tau^p \neq \tau^w$ since the former is metrizable while the latter is not, and $\tau^w\neq \tau^n$. $\mathsf{I}$-hypercyclicity of an operator (and their variants) with respect to the topologies $\tau^n, \tau^w, \tau^p$ will be also referred to as norm, weakly, and pointwise $\mathsf{I}$-hypercyclicity, respectively. We are going to show that these notions coincide for weighted backward shifts on $\ell_p$, provided that $\mathsf{I}$ 
is a \textquotedblleft well-behaved\textquotedblright\, $F_\sigma$-ideal. 
This extends the known equivalence between (ordinary) norm hypercyclicity and weakly hypercyclicity of weighted backward shifts, see 
\cite[Proposition 10.19]{MR2533318}.

Recall also that an ideal $\mathsf{I}$ is $\omega$ is \emph{countably generated} if there exists a sequence $(S_j)$ of subsets of $\omega$ such that $S \in \mathsf{I}$ if and only if $S\subseteq \bigcup_{j \in F}S_j$ for some $F \in \mathrm{Fin}$. It is known that countably generated ideals are $Q^+$\emph{-ideals}, that is, for every $S \in \mathsf{I}^+$ and every partition $(F_n)$ of $S$ into finite sets, there exists $T \in \mathsf{I}^+$ such that $|T\cap F_n|\le 1$ for all $n \in \omega$, see \cite{MR3351990, MR3666945}. 
Countably generated ideals are precisely those which are isomorphic to one of the following: $\mathrm{Fin}$ or the Fubini product $\mathrm{Fin}\times \emptyset$ or the Fubini sum $\mathrm{Fin}\oplus \mathcal{P}(\omega)$;  see \cite[Proposition 1.2.8]{MR1711328} and \cite[Section 2]{MR3543775}. 

\begin{thm}\label{thm:equivalenceweaklynorm}
    Let $\mathsf{I}$ be 
    a countably generated ideal on $\omega$. 
    Let $\bm{w}$ be a bounded sequence of positive reals and consider the weighted backward shift $B_{\bm{w}}$ on $c_0$ or on $\ell_p$, with $p \in [1,\infty)$. 
  
    Then the following are equivalent\textup{:}
    \begin{enumerate}[label={\rm (\roman{*})}]
    \item \label{item:1equivalenthyperc} $B_{\bm{w}}$ is norm $\mathsf{I}$-hypercyclic\textup{;}
    \item \label{item:2equivalenthyperc} $B_{\bm{w}}$ is weakly $\mathsf{I}$-hypercyclic\textup{;}
    \item \label{item:3equivalenthyperc} $B_{\bm{w}}$ is pointwise $\mathsf{I}$-hypercyclic\textup{.}
    \end{enumerate}
    %
\end{thm}

However, it is known that the sets of norm $\mathsf{I}$-hypercyclic vectors and weakly $\mathsf{I}$-hypercyclic vectors do not necessarily coincide even for $\mathsf{I}=\mathrm{Fin}$: in fact, if $\inf\{w_n: n \in \omega\}>1$ and $p \in (1,\infty)$, then there exists a weakly hypercyclic vector for $B_{\bm{w}}$ which is not norm hypercyclic, see \cite[Theorem 4.2]{MR2091459}. On a different direction, the equivalence between norm hypercyclicity and weakly hypercyclicity fails for bilateral weighted backward shifts, see \cite[Corollary 3.3]{MR2091459}.

As pointed out in Remark \ref{rmk:onlycountablygenerated} below, the proof of Theorem \ref{thm:equivalenceweaklynorm} cannot be adapted for the ideal $\mathsf{Z}$ (which is not countably generated). The next result shows that the equivalence between norm $\mathsf{Z}$-hypercyclicity and pointwise $\mathsf{Z}$-hypercyclicity fails in general. 

\begin{thm}\label{thm:nonequivalence}
For each $p \in [1,\infty)$, there exists a decreasing real sequence $\bm{w}$ with $\lim_nw_n=1$ for which the weighted backward shift $B_{\bm{w}}$ is pointwise upper frequently hypercyclic, norm hypercyclic, and not norm upper frequently hypercyclic. 
\end{thm}

We leave as open question for the interested reader to check the topological complexities of the sets of norm upper frequently hypercyclic vectors and weakly upper frequently hypercyclic vectors of [unilateral] weighted backward shifts. Lastly, we ask also which subsets of a given underlying space $X$ can be attained as sets of upper frequently hypercyclic operators of some continuous linear operator on $X$.



\section{Proofs
 and Related Results}\label{sec:proof22}

 Let us recall that a \emph{lower semicontinuous submeasure} (in short, \emph{lscsm}) is a map $\varphi: \mathcal{P}(\omega) \to [0,\infty]$ such that: 
\begin{enumerate}[label={\rm (\roman{*})}]
\item 
$\varphi(\emptyset)=0$, 
\item
$\varphi(A)\le \varphi(B)$ for all $A\subseteq B\subseteq \omega$,
\item
$\varphi(A\cup B)\le \varphi(A)+\varphi(B)$ for all $A,B\subseteq \omega$, 
\item
$\varphi(F)<\infty$ for all $F \in \mathrm{Fin}$, and 
\item
$\varphi(A)=\sup\{\varphi(A\cap F): F \in \mathrm{Fin}\}$ for all $A\subseteq \omega$. 
\end{enumerate} 

By a classical result of Mazur, see \cite{MR1124539}, an ideal $\mathsf{I}$ on $\omega$ is a $F_\sigma$ ideal if and only if there exists a lscsm $\varphi$ such that 
\begin{equation}\label{eq:mazur}
\mathsf{I}:=\mathrm{Fin}(\varphi):=\left\{S\subseteq \omega: \varphi(S)=\infty\right\}
\quad \text{ and }\quad 
\varphi(\omega)=\infty.
\end{equation}

By another classical result due to Solecki, see \cite[Theorem 3.1]{MR1708146}, an ideal $\mathsf{I}$ on $\omega$ is an analytic $P$-ideal if and only if there exists a lscsm
$\varphi$ 
such that
\begin{equation}\label{eq:characterizationanalPideal}
\mathsf{I}=\mathrm{Exh}(\varphi):=\{S\subseteq \omega: \|S\|_\varphi=0\} 
\,\,\,\text{ and }\,\,\,
0<\|\omega\|_\varphi \le \varphi(\omega)<\infty,
\end{equation}
where 
$$
\|S\|_\varphi:=\inf\{\,\varphi(S\setminus F): F \in \mathrm{Fin}\}
$$
for all $S\subseteq \omega$. In particular, every analytic $P$-ideal is necessarily $F_{\sigma\delta}$. 
Note that $\|\cdot\|_\varphi$ is a monotone, subadditive, and  invariant modulo finite sets. By the monotonicity of $\varphi$, the value $\|S\|_\varphi$ coincides with $\lim_n \varphi(S\setminus [0,n])$, hence it represents the $\varphi$-mass at infinity of the set $S\subseteq \omega$. However, the choice of lscsm is not unique: e.g., let $\varphi$ and $\nu$ be the lscsms defined by 
$\varphi(S):=\sup_n |S\cap [0,n]|/(n+1)$ and $\nu(S):=\sup_n |S \cap [2^n,2^{n+1})|/2^n$ for all $S\subseteq \omega$. 
Then 
$
\mathsf{Z}=\mathrm{Exh}(\varphi)=\mathrm{Exh}(\nu),
$ 
see 
\cite[Theorem 1.13.3(a)]{MR1711328}. 

\subsection{Upper bounds on topological complexities.} In the next results, a subset $\mathcal{S}\subseteq \mathcal{P}(\omega)$ is said to be \emph{hereditary} if $A \in \mathcal{S}$ and $B\subseteq A$ implies $B \in \mathcal{S}$.

\begin{prop}\label{prop:maintooltopological}
Let $T: X\to X$ be a continuous function, where $X$ is a second countable topological space. 
Let also $\mathcal{F}\subseteq \mathcal{P}(\omega)$ be a hereditary closed set. Then 
$$
\left\{x \in X: \{n \in \omega: T^n x \in U\}\notin \mathcal{F}
\text{ for all nonempty open }U\subseteq X
\right\}
$$
is a $G_\delta$-set. 
\end{prop}
\begin{proof}
    Fix a nonempty open $U\subseteq X$, set 
    $
    Y_U:=\left\{x \in X: \{n \in \omega: T^n x \in U\}\notin \mathcal{F}\right\},
    $ 
    and define $\mathcal{G}:=\mathcal{P}(\omega)\setminus \mathcal{F}$. Since $\mathcal{F}$ is hereditary closed, $\mathcal{G}$ is open and
    $$
    \forall A,B\subseteq \omega, \quad 
    (A\subseteq B \text{ and }A \in \mathcal{G}) \implies B \in \mathcal{G}.
    $$
    Let $(\mathcal{G}_k: k \in \omega)$ be a sequence of basic clopen sets of $\mathcal{P}(\omega)$ such that $\mathcal{G}=\bigcup_k \mathcal{G}_k$. For each $k \in \omega$, there exists $F_k \in \mathrm{Fin}$ such that $\mathcal{G}_k=\{S\subseteq \omega: S \cap [0,\max F_k]=F_k\}$. Accordingly, define the open set 
    $
    \widehat{\mathcal{G}}_k:=\{S\subseteq \omega: F_k\subseteq S \cap [0,\max F_k]\}. 
    $ 
    Then, it is routine to check that $\mathcal{G}=\bigcup_k  \widehat{\mathcal{G}}_k$. At this point, notice that
    $$
    Y_U=\bigcup_{k\in \omega}Y_{k}, 
    \quad \text{ where}\quad 
    Y_{k}:=\left\{x \in X: \{n \in \omega: T^n x \in U\}\in \widehat{\mathcal{G}}_k\right\}. 
    $$
    Now, it is easy to check that each $Y_{k}$ is open: suppose that we can pick $x \in Y_k$. Then $T^nx \in U$ for all $n \in F_k$. Since each $T^n$ is continuous, there exists an open neighborhood $V$ of $x$ such that $T^n[V]\subseteq U$ for all $n \in F_k$. Hence $V\subseteq Y_k$. This implies that $Y_k$ is open, hence $Y_U$ is open. By hypothesis, there exists a countable base of open sets $(U_m)$. The conclusion follows by the fact that the claimed set can be rewritten as $\bigcap_m Y_{U_m}$. 
    \end{proof}

Hereafter, we write also 
$$
\hat{\Pi}^0_1:=\left\{\mathcal{S}\subseteq \mathcal{P}(\omega): \mathcal{S} \text{ is hereditary closed}\right\},
$$
and we define, recursively for positive integers $k$, the modified versions of Borel pointclasses by 
$$
\hat{\Sigma}^0_{2k}:=\left\{\bigcup\nolimits_{j \in \omega}\mathcal{S}_j:\, \forall j \in \omega, \mathcal{S}_j \in \hat{\Pi}^0_{2k-1} 
\right\}
$$
and 
$$
\hat{\Pi}^0_{2k+1}:=\left\{\bigcap\nolimits_{j \in \omega}\mathcal{S}_j:\, \forall j \in \omega, \mathcal{S}_j \in \hat{\Sigma}^0_{2k}. 
\right\}
$$
For instance, $\hat{\Sigma}^0_2$ is the family of subsets $\mathcal{S}\subseteq \mathcal{P}(\omega)$ which can be written as countable union of hereditary closed sets, $\hat{\Pi}^0_3$ is the family of all $\mathcal{S}\subseteq \mathcal{P}(\omega)$ which can be written as countable union of countable intersection of hereditary closed sets, etc.; we will use also the standard notation of Borel classes $\Sigma^0_k(X)$ and $\Pi^0_k(X)$ as in \cite[Chapter 22]{MR1321597}. 

Given a hereditary subset $\mathcal{F}\subseteq \mathcal{P}(\omega)$ (in place of an ideal on $\omega$) and a continuous function $T$ on a topological space $X$ (in place of a continuous linear operator on a topological vector space), it still makes sense to write 
$$
\mathrm{HC}_T(\mathcal{F})
$$
for the set of all points $x \in X$ such that $\{n \in \omega: T^nx \in U\} \notin \mathcal{F}$ for all nonempty open sets $U\subseteq X$; equivalently, this is the set of hypercyclic points of $T$ with respect to the Furstenberg family $\mathcal{P}(\omega)\setminus \mathcal{F}$, as defined in \cite{MR3847081}.

\begin{thm}\label{thm:mainnewtopological}
    Let $T: X\to X$ be a continuous function, where $X$ is a second countable topological space and pick a subset $\mathcal{F}\subseteq \mathcal{P}(\omega)$. Then the following hold\textup{:}
    \begin{enumerate}[label={\rm (\roman{*})}]
\item \label{item:thmtopol01} If $\mathcal{F} \in \hat{\Pi}^0_{1}$ then $\mathrm{HC}_T(\mathcal{F}) \in \Pi^0_2(X)$\textup{;}
\item \label{item:thmtopol02} If $\mathcal{F} \in \hat{\Pi}^0_{2k+1}$ for some $k\ge 1$ then $\mathrm{HC}_T(\mathcal{F}) \in \Sigma^0_{2k+1}(X)$\textup{;} 
\item \label{item:thmtopol03} If $\mathcal{F} \in \hat{\Sigma}^0_{2k}$ for some $k\ge 1$ then $\mathrm{HC}_T(\mathcal{F}) \in \Pi^0_{2k}(X)$\textup{.} 
\end{enumerate}
\end{thm}
\begin{proof}
    \ref{item:thmtopol01}. This is just a rewriting of Proposition \ref{prop:maintooltopological}. Items \ref{item:thmtopol02} and \ref{item:thmtopol03} are obtained by induction on $k$. For instance, if $\mathcal{F} \in \hat{\Sigma}^0_2$, then $\mathcal{F}=\bigcup_j \mathcal{F}_j$ for some sequence $(F_j)$ in $ \hat{\Pi}^0_1$. Taking into account that each $\mathrm{HC}_T(\mathcal{F}_j)$ is $G_\delta$ (i.e., belongs to $\Pi^0_2(X)$) and 
    $$
    \mathrm{HC}_T(\mathcal{F})=\bigcap\nolimits_{j \in \omega} \mathrm{HC}_T(\mathcal{F}_j),
    $$
    it follows that $\mathrm{HC}_T(\mathcal{F}) \in \Pi^0_2(X)$. Similarly, if $\mathcal{F}=\bigcap_j \mathcal{F}_j \in \hat{\Pi}^0_3$ for some $(F_j)$ in $ \hat{\Sigma}^0_2$, we obtain with a similar reasoning $\mathrm{HC}_T(\mathcal{F})=\bigcup_j \mathrm{HC}_T(\mathcal{F}_j) \in \Pi^0_3(X)$, and so on. 
\end{proof}

As first consequence, we obtain an alternative proof of Theorem \ref{thm:startingthm}\ref{item:02}:
\begin{cor}\label{cor:Fsigma}
Let $T: X\to X$ be a continuous linear operator, where $X$ is a 
second countable topological vector space, 
and let $\mathsf{I}$ be a $F_\sigma$-ideal on $\omega$. Then $\mathrm{HC}_T(\mathsf{I})$ is a $G_{\delta}$-set. 
\end{cor}
\begin{proof}
    Pick a sequence $(\mathcal{S}_j)$ of closed subsets of $\mathcal{P}(\omega)$ such that $\mathsf{I}=\bigcup_j \mathcal{S}_j$. For each $j \in \omega$, define $\mathcal{F}_j:=\{A\subseteq \omega: A\subseteq B \text{ for some }B \in \mathcal{S}_j\}$. Then each $\mathcal{F}_j$ is hereditary closed and $\mathsf{I}=\bigcup_j \mathcal{F}_j$, i.e., $\mathsf{I} \in \hat{\Sigma}^0_2$. Therefore $\mathrm{HC}_T(\mathsf{I}) \in \Pi^0_2(X)$  
    by Theorem \ref{thm:mainnewtopological}\ref{item:thmtopol03}. 
\end{proof}

As another consequence, we obtain a proof of Theorem \ref{thm:topologicalanalyticP}:
\begin{proof}
    [Proof of Theorem \ref{thm:topologicalanalyticP}]
    It is well known, using the representation \eqref{eq:characterizationanalPideal}, that every analytic $P$-ideal $\mathsf{I}$ belongs to $\hat{\Pi}^0_3$, see e.g. \cite[p. 201]{MR2849045}. 
    The conclusion follows by Theorem \ref{thm:mainnewtopological}\ref{item:thmtopol02}. 
\end{proof}

\begin{rmk}
It is a open problem to establish whether every $F_{\sigma\delta}$-ideal on $\omega$ belongs to the class $\hat{\Pi}^0_3$ (such ideals are commonly known as \emph{weakly Farah}), see e.g. \cite[Question 5]{MR2849045}. Accordingly, in case of positive answer of the latter conjecture, it would follow that if $\mathsf{I}$ is a $F_{\sigma\delta}$-ideal on $\omega$ then $\mathrm{HC}_T(\mathsf{I})$ is $G_{\delta\sigma}$. 
\end{rmk}



\medskip

\subsection{Lower bounds on topological complexities.} Before we proceed to the proof of Theorem \ref{thm:mainnegativekarl}, we recall the notion of frequently hypercyclicity (which is stronger than upper frequently hypercyclicity); see e.g. \cite[Chapter 9]{MR2919812}. To this aim, for each $n \in \omega$, define the lscsm $\mu_n: \mathcal{P}(\omega) \to [0,\infty]$ by 
$$
\forall S\subseteq \omega, \quad 
\mu_n(S):=\frac{|S\cap [0,n)]}{n+1}.
$$
Let also $\mathsf{d}^\star$ and $\mathsf{d}_\star$ be the upper and lower asymptotic density on $\omega$, respectively, i.e., 
$$
\forall S\subseteq \omega, \quad 
\mathsf{d}^\star(S):=\limsup_{n\to \infty}\mu_n(S) 
\quad \text{ and }\quad 
\mathsf{d}_\star(S):=\liminf_{n\to \infty}\mu_n(S). 
$$
In particular, it is immediate that $\mathsf{Z}=\{S\subseteq \omega: \mathsf{d}^\star(S)=0\}$, hence a continuous map $T: X\to X$ is upper frequently hypercyclic if and only if there exists $x \in X$ such that $\mathsf{d}^\star(\{n \in \omega: T^nx \in U\})>0$ for all nonempty open sets $U\subseteq X$. 
\begin{defi}
Let $T:X\to X$ be a continuous map, where $X$ is a topological space. Then $T$ is \textbf{frequently hypercyclic} if there exists $x \in X$ such that 
$$
\mathsf{d}_\star(\{n \in \omega: T^nx \in U\})>0
$$
for all nonempty open sets $U\subseteq X$. 
\end{defi}

Of course, also the above definition depends on the underlying topology on $X$, hence it makes sense to speak about norm frequently hypercyclicity, pointwise frequently hypercyclicity, etc. 
Lastly, we write $s^\frown t$ to denote the concatenation of two sequences $s\in \mathbf{R}^{<\omega}$ and $t\in \mathbf{R}^{<\omega}\cup \mathbf{R}^\omega$, and $0^n:=(0,0,\ldots,0)$ where $0$ is repeated $n$ times.

\begin{thm}\label{thm:main}
    Let $\bm{w}$ be a bounded sequence of reals with $w_n\ge 1$ for all $n \in \omega$ and consider the weighted backward shift $B_{\bm{w}}$ on $c_0$ or $\ell_p$, with $p \in [1,\infty)$, endowed with the product topology. Suppose also that the sequence $B_{\bm{w}}$ is pointwise frequently hypercyclic. 
 
 Then $\mathrm{UFHC}(B_{\bm{w}})$ is not a $F_{\sigma\delta}$-set.  
\end{thm}
\begin{proof}
Suppose that the underlying space $X$ is $c_0$ or $\ell_p$, with $p \in [1,\infty)$, endowed with the product topology. Suppose also that there exists a pointwise frequently hypercyclic vector $y=(y_n: n \in \omega) \in X$ of the weighted backward shift $B_{\bm{w}}$ on $X$. We need to show that $\mathrm{UFHC}(B_{\bm{w}})$ is not a $F_{\sigma\delta}$-subset of $X$.  

To this aim, consider the Baire space $\omega^\omega$ (which is a zero-dimensional Polish space, once it is endowed with the product topology of the discrete topology on $\omega$) and define 
$$
C_3:=\{x \in \omega^\omega: \lim_{n\to \infty} x_n=+\infty\}.
$$
It is well known that $C_3$ is $F_{\sigma\delta}$ but not $G_{\delta\sigma}$, see \cite[Definition 22.9 and Exercise 23.2]{MR1321597}. Next, define also the closed set 
$$
\Delta:=\{x \in \omega^\omega: x_n\le n \text{ for all }n \in \omega\},
$$
and observe that $\Delta$ is a Polish space on its own by Alexandrov's theorem, see e.g. \cite[Proposition 3.7 and Theorem 3.11]{MR1321597}.
\begin{claim}\label{claim:complexityD}
    $D:=\Delta\setminus C_3$ is $G_{\delta\sigma}$ but not $F_{\sigma\delta}$ in $\Delta$. 
\end{claim}
\begin{proof}
    Notice that $\Delta\setminus D=\Delta \cap C_3$ is the intersection of two $F_{\sigma\delta}$ sets in $\omega^\omega$, hence also in $\Delta$. Hence $D$ is $G_{\delta\sigma}$ in $\Delta$. Moreover, the map $h: \omega^\omega\to \Delta$ defined by $h(x):=(\min\{n,x_n\}: n \in \omega)$ is continuous and $h^{-1}[\Delta \cap C_3]=C_3$. This implies that $\Delta \cap C_3$ is not $G_{\delta\sigma}$ in $\Delta$, hence its complement $D$ is not $F_{\sigma\delta}$ in $\Delta$. 
\end{proof}    

Our proof strategy will be the construction of a continuous function 
$$
f: \Delta\to X
$$
such that $f^{-1}[\mathrm{UFHC}(B_{\bm{w}})]=D$. In fact, the continuity of $f$ would imply, thanks to Claim \ref{claim:complexityD}, that $\mathrm{UFHC}(B_{\bm{w}})$ is not a $F_{\sigma\delta}$-set, which will conclude the proof. 

\medskip

Let $(V_n: n \in \omega)$ be a countable base of nonempty open sets of $X$ and, for each $n \in \omega$, pick $k_n \in \omega$ such that $V_n= \{x \in X: x_0 \in U_0, x_1 \in U_1, \ldots, x_{k_n} \in U_{k_n}\}$ for some open sets $U_0, U_1, \ldots, U_{k_n}\subseteq \mathbf{R}$. Without loss of generality, we can assume that the sequence $(k_n: n \in \omega)$ is increasing. 
For each $t \in \omega$, define also the set of integers 
$$
S_t:=\{n \in \omega: T^ny \in V_{t}\}.
$$
Since $y$ is pointwise frequently hypercyclic, it follows that $\mathsf{d}_\star(S_t)>0$ for all $t \in \omega$. For each $n,k \in \omega$ with $n\le k$, define 
$$
\tilde{w}_{n,k}:=w_nw_{n+1}\cdots w_{k},
$$
and fix a real sequence $(\varepsilon_n) \in (0,1)^\omega$ such that $\lim_n \varepsilon_n=0$. 

\medskip

\textsc{Construction of the function $f$.} Fix a sequence $x \in \Delta$, and define the components of $f(x)$ recursively as it follows. 
\begin{enumerate}[label={\rm (\roman{*})}]
\item Pick an integer $m_0 > k_{0}$ such that 
$$
\forall m\ge m_0, \quad 
\mu_{m}(S_{0}) \ge (1-\varepsilon_0)\, \mathsf{d}_\star(S_{0}). 
$$
Note that this is indeed possible. 
Set $\hat{m}_0:=m_0+k_{0}$ and define the finite sequence 
$$
s^{(0)}:=0^{(x_0+1)\hat{m}_0} \text{ }^\frown \left( \frac{y_n}{\tilde{w}_{n+1,(x_0+1)\hat{m}_0+n}}: n \in [0,\hat{m}_0) \right). 
$$

\item Suppose that, for some $t \in \omega$, the integer $m_{t} \in \omega$ and the finite sequence $s^{(t)}\in \mathbf{R}^{<\omega}$ have been defined, with $\hat{m}_t:=m_t+k_{t}$. For convenience, set also 
$$
\alpha_t:=\sum_{j=0}^t \hat{m}_j, \,\, 
\beta_t:=\sum_{j=0}^t (x_j+1)\hat{m}_j,\,\,\text{ and }\,\,
\gamma_t:=\sum_{j=0}^t (x_j+2)\hat{m}_j.
$$

\item 
Pick an integer $m_{t+1} > \max\{k_{t+1}, (t+1)^2\alpha_t\}$ such that 
$$
\forall m\ge m_{t+1}, \forall j \in [0,t+1],\quad 
\mu_{\alpha_t+m}(S_{j}\setminus [0,\alpha_t)) \ge  (1-\varepsilon_{t+1}) \,\mathsf{d}_\star(S_{j}). 
$$
Note that this is again possible. 
Set $\hat{m}_{t+1}:=m_{t+1}+k_{t+1}$ and define 
\begin{equation}\label{eq:defst}
s^{(t+1)}:=0^{(x_{t+1}+1)\hat{m}_{t+1}}\text{ }^\frown \left(\frac{y_n}{\tilde{w}_{n+1,\beta_t+(x_{t+1}+1)\hat{m}_{t+1}+n}}: n \in [\alpha_t, \alpha_t+\hat{m}_{t+1})\right).
\end{equation}
\end{enumerate}

Accordingly, for each $t \in \omega$, the finite sequence $s^{(t)}$ has lenght $(x_t+2)\hat{m}_t$, hence $s^{(0)}\text{ }^\frown s^{(1)}\text{ }^\frown \cdots \text{ }^\frown s^{(t)}$ has lenght $\gamma_t$. Finally, define the sequence $f(x)$ by 
$$
f(x):=s^{(0)}\text{ }^\frown s^{(1)}\text{ }^\frown s^{(2)}\text{ }^\frown \ldots 
$$

\medskip

\begin{claim}\label{claim:continuity}
    $f$ is well defined and continuous. 
\end{claim}
\begin{proof}
Suppose that $X=\ell_p$ (the case $X=c_0$ being similar). As it follows by \eqref{eq:defst}, each $s^{(t)}$ contains a block of components of $y$, each one divided by values $\ge 1$. Since the norm of a sequence is invariant under the addition of zeros, we get $\|f(x)\|\le \|y\|$ for each $x \in \Delta$. Hence $f(x) \in X$ for each $x \in \Delta$, that is, $f$ is well defined. 

For the second claim, we show something stronger, namely, $f$ is norm continuous. Fix $\varepsilon>0$ and pick $n_0 \in \omega$ such that $\|0^n\text{ }^\frown (y_n,y_{n+1}, \ldots)\|< \varepsilon/2$ for all $n\ge n_0$. Pick 
$x,x^\prime \in \Delta$ such that $x_n=x^\prime_n$ for all 
$n\in [0, n_0]$. 
Taking again into account that the denominators in \eqref{eq:defst} are $\ge 1$, it follows that 
\begin{displaymath}
    \begin{split}
        \|f(x)-f(x^\prime)\|&\le 2\|0^n\text{ }^\frown (y_n,y_{n+1}, \ldots)\|<\varepsilon. 
    \end{split}
\end{displaymath}
In particular, $f$ is pointwise continuous. 
\end{proof}

At this point, for each $x \in \Delta$, define the map $\iota_x: \omega\to \omega$ such that, if $n=\gamma_{t-1}+(x_t+1)\hat{m}_t+u$ for some $u \in [0,m_t)$ then $\iota_x(n):=\alpha_{t-1}+u$ (where, by convention, we assume that $\alpha_{-1}:=\beta_{-1}:=\gamma_{-1}:=0$); in all remaining cases $\iota_x(n):=0$. 

\begin{claim}\label{claim:equivalence}
Fix $x \in \Delta$ and $i \in \omega$. Then
$$
T^nf(x) \in V_{i}
\quad \text{ if and only if }\quad 
T^{\iota_x(n)}y \in V_{i}
$$
for all sufficiently large $n \in \omega$ with $\iota_x(n)\neq 0$. 
\end{claim}
\begin{proof}
Set $\iota:=\iota_x$, and recall that $m_j>k_j$ for all $j \in \omega$ and $(k_j: j \in \omega)$ is increasing. Pick an integer $n \in \omega$ with $\iota(n)\neq 0$. Hence there exist $t,u \in \omega$ such that $n=\gamma_{t-1}+(x_t+1)\hat{m}_t+u$ and $u \in [0,m_t)$. Notice that by construction $n\ge \hat{m}_t \ge 1$. In addition, suppose that $t\ge i$ (hence, we remove only finitely many $n$, and $m_t>k_t\ge k_i$). 

    On the one hand, we have 
\begin{displaymath} 
\begin{split}
T^{\iota(n)}y&=(w_1w_2\cdots w_{\iota(n)}y_{\iota(n)}, w_2w_3\cdots w_{{\iota(n)}+1}y_{{\iota(n)}+1}, \ldots)\\
&=(w_1w_2\cdots w_{\alpha_{t-1}+u}y_{\alpha_{t-1}+u}, w_2w_3\cdots w_{{\alpha_{t-1}+u}+1}y_{{\alpha_{t-1}+u}+1}, \ldots).
\end{split}
    \end{displaymath}   
In particular, for all integers $j \in [0,k_{i}]$, we get 
$$
(T^{\iota(n)}y)_{j}
=\tilde{w}_{1+j,\alpha_{t-1}+u+j}y_{\alpha_{t-1}+u+j}.
$$
    
    On the one hand, setting $z:=f(x)$, we have that 
\begin{displaymath} 
    T^nz =(w_1w_2\cdots w_nz_n, w_2w_3\cdots w_{n+1}z_{n+1}, \ldots).
    \end{displaymath}   
It follows that, for each integer $j \in [0,k_{i}]$, we have
\begin{displaymath} 
\begin{split}
(T^nz)_j&=w_{1+j}w_{2+j}\cdots w_{n+j}z_{n+j}\\
&=w_{1+j}w_{2+j}\cdots w_{\gamma_{t-1}+(x_t+1)\hat{m}_t+u+j} z_{\gamma_{t-1}+(x_t+1)\hat{m}_t+u+j}\\
&=w_{1+j}w_{2+j}\cdots w_{\gamma_{t-1}+(x_t+1)\hat{m}_t+u+j} \cdot \frac{y_{\alpha_{t-1}+u+j}}{\tilde{w}_{\alpha_{t-1}+u+j+1,\beta_{t-1}+(x_t+1)\hat{m}_{t}+\alpha_{t-1}+u+j}}\\
&=w_{1+j}w_{2+j}\cdots w_{\gamma_{t-1}+(x_t+1)\hat{m}_t+u+j} \cdot \frac{y_{\alpha_{t-1}+u+j}}{\tilde{w}_{\alpha_{t-1}+u+j+1,\gamma_{t-1}+(x_t+1)\hat{m}_{t}+u+j}}\\
&=w_{1+j}w_{2+j}\cdots w_{\alpha_{t-1}+u+j}y_{\alpha_{t-1}+u+j}
\end{split}
    \end{displaymath}  
Therefore $(T^nz)_j=(T^{\iota(n)}y)_j$ for all $j \in [0,k_i]$.    
\end{proof}

\begin{claim}\label{claim:inclusion1}
$D\subseteq f^{-1}[\mathrm{UFHC}(B_{\bm{w}})]$. 
\end{claim}
\begin{proof}
Fix $x \in D$, so that $x \in \omega^\omega$ satisfies $\liminf_n x_n<\infty$ and $x_n\le n$ for all $n \in \omega$. In particular, there exist $j\in \omega$ and an infinite set $Q\subseteq \omega$ such that $x_n=j$ for all $n \in Q$. At this point, it is enough to show that the sequence $z:=f(x)$ is (pointwise) upper frequently hypercyclic for $B_{\bm{w}}$. To this aim, fix an integer $i \in \omega$. 
Thanks to Claim \ref{claim:equivalence}, there exists $F\in \mathrm{Fin}$ such that 
$$
A\setminus F\subseteq \{n \in \omega: T^nz \in V_i\}, 
\quad \text{ where }\,\, A:=\{n \in \iota_x^{-1}((0,\infty)): T^{\iota_x(n)}y \in V_i\}.
$$
Now, recall that $\mathsf{d}_\star(S_i)=\mathsf{d}_\star(\{n \in \omega: T^ny \in V_i\})>0$ since $y$ is pointwise frequently hypercyclic. It follows by the construction of $f(x)$ that 
$$
\mathsf{d}^\star\left(\bigcup\nolimits_{t \in Q\setminus [0,i]}(S_i \cap [\alpha_{t-1},\alpha_{t-1}+m_t))\right)\ge \mathsf{d}_\star(S_i)>0.
$$
Moreover, by the definition of $\iota_x$, we have 
$$
\iota_x^{-1}\left[S_i \cap [\alpha_{t-1},\alpha_{t-1}+m_t)\right]=(S_i \cap [\alpha_{t-1},\alpha_{t-1}+m_t))+\beta_{t-1}+(x_t+1)\hat{m}_t\subseteq A.  
$$
for all integers $t>i$. This implies, by the above observations of the standard properties of $\mathsf{d}^\star$, see e.g. \cite{MR4054777}, that 
\begin{displaymath}
    \begin{split}
        \mathsf{d}^\star(\{n \in \omega: T^nz \in V_i\})&\ge \mathsf{d}^\star(A\setminus F)=\mathsf{d}^\star(A) \\
        &\hspace{-20mm}\ge \mathsf{d}^\star\left(\bigcup\nolimits_{t \in Q\setminus [0,i]}(S_i \cap [\alpha_{t-1},\alpha_{t-1}+m_t))+\beta_{t-1}+(x_t+1)\hat{m}_t\right)\\
        &\hspace{-20mm}\ge \mathsf{d}^\star\left(\bigcup\nolimits_{t \in Q\setminus [0,i]}(S_i \cap [\alpha_{t-1},\alpha_{t-1}+m_t))+\sum_{u<t}(u+1)\hat{m}_u+(j+1)\hat{m}_t\right)\\
        &\hspace{-20mm}\ge \mathsf{d}^\star\left(\bigcup\nolimits_{t \in Q\setminus [0,i]}(S_i \cap [\alpha_{t-1},\alpha_{t-1}+m_t))+(j+2)\hat{m}_t\right)\\
        &\hspace{-20mm}\ge 
        \limsup_{t \to \infty, t \in Q}\frac{|S_i \cap [\alpha_{t-1}, \alpha_{t-1}+m_t))|}{m_t+(j+2)\hat{m}_t}\\
        &\hspace{-20mm}=
        \limsup_{t \to \infty, t \in Q}\mu_{\alpha_{t-1}+m_t}(S_i \setminus [0,\alpha_{t-1})) \cdot \frac{\alpha_{t-1}+m_t}{m_t+(j+2)\hat{m}_t}\\
        &\hspace{-20mm}
        \ge \mathsf{d}_\star(S_i)\cdot \liminf_{t \to \infty}\frac{\alpha_{t-1}+m_t}{m_t+(j+2)\hat{m}_t} \\
        &\hspace{-20mm}
        \ge \mathsf{d}_\star(S_i)\cdot \liminf_{t \to \infty}\frac{1}{1+(j+2)\hat{m}_t/m_t} \\
        &\hspace{-20mm}
        \ge \mathsf{d}_\star(S_i)\cdot \frac{1}{1+2(j+1)}>0. 
    \end{split}
\end{displaymath}
Therefore $z \in \mathrm{UFHC}(B_{\bm{w}})$.  
\end{proof}

\begin{claim}\label{claim:inclusion2}
$f^{-1}[\mathrm{UFHC}(B_{\bm{w}})]\subseteq D$.
\end{claim}
\begin{proof}
Fix $x \in \Delta\setminus D$,  so that $x \in \omega^\omega$ satisfies $\lim_n x_n=\infty$ and $x_n\le n$ for all $n \in \omega$. We claim that $\mathsf{d}_\star(\{n \in \omega: z_n=0\})=1$, where $z:=f(x)$. In fact, taking into account the above construction (so that, for each $t \in \omega$, we have $z_n=0$ for all $n \in [\gamma_{t-1}, \gamma_{t-1}+(x_t+1)\hat{m}_t)$, and $\gamma_{t-1}\le \hat{m}_t$), we obtain  
\begin{displaymath}
    \begin{split}
        \mathsf{d}_\star(\{n \in \omega: z_n=0\})
        &\ge \liminf_{t \to \infty} \, \min_{i \in [\gamma_{t-1}+(x_t+1)\hat{m}_t, \gamma_t)} \, \mu_i(\{n \in \omega: z_n=0\})\\
        &\ge \liminf_{t \to \infty} \, \frac{(x_t+1)\hat{m}_t}{\gamma_{t-1}+(x_t+2)\hat{m}_t}\\
        &\ge \liminf_{t \to \infty} \, \frac{(x_t+1)\hat{m}_t}{\hat{m}_t+(x_t+2)\hat{m}_t}=1.
    \end{split}
\end{displaymath}
In particular, we get $z\notin \mathrm{UFHC}(B_{\bm{w}})$. 
\end{proof}

At this point, the identity $D=f^{-1}[\mathrm{UFHC}(B_{\bm{w}})]$ follows putting together Claim \ref{claim:inclusion1} and Claim \ref{claim:inclusion2}. This concludes the proof. 
\end{proof}

Now, we recall the following characterization by Bayart and Rusza \cite{MR3334899}:
\begin{thm}\label{thm:bayartRuzsa}
    Let $\bm{w}$ be a bounded sequence of positive reals and consider the weighted backward shift $B_{\bm{w}}$ on $\ell_p$, with $p \in [1,\infty)$. Then the following are equivalent\textup{:}
    \begin{enumerate}[label={\rm (\roman{*})}]
    \item $B_{\bm{w}}$ is norm frequently hypercyclic\textup{;}
    \item $B_{\bm{w}}$ is norm upper frequently hypercyclic\textup{;}
\item condition \eqref{eq:conditionufhc} holds\textup{.}
    \end{enumerate}
 \end{thm}
\begin{proof}
    See \cite[Theorem 4]{MR3334899}. 
\end{proof}

Putting together the above results, we obtain a proof of Theorem \ref{thm:mainnegativekarl}. 
\begin{proof}
    [Proof of Theorem \ref{thm:mainnegativekarl}] 
    Thanks to Theorem \ref{thm:bayartRuzsa}, the weighted backward shift $B_{\bm{w}}$ is norm frequently hypercyclic. Hence, it is also pointwise frequently hypercyclic. The conclusion follows by Theorem \ref{thm:main}. The second part is immediate by the fact that constant sequence $(\lambda, \lambda, \ldots)$ satisfies \eqref{eq:conditionufhc} if $\lambda>1$. 
\end{proof}

In particular, we have the following consequence: 
\begin{cor}
    Consider the unilateral backward shift $B$ on $\ell_2$, endowed with the product topology. Then $\mathrm{UFHC}(2B)$ is $G_{\delta\sigma}$ but not $F_{\sigma\delta}$. 
\end{cor}
\begin{proof}
    It follows by Theorem \ref{thm:topologicalanalyticP} and Theorem \ref{thm:mainnegativekarl}. 
\end{proof}

\medskip

\subsection{[Non-]equivalence of $\mathsf{I}$-hypercyclicity for different topologies} Finally, we proceed to the proofs of Theorem \ref{thm:equivalenceweaklynorm} and Theorem \ref{thm:nonequivalence}. Here, if $y \in \ell_p$ and $S\subseteq \omega$, we write $y \upharpoonright S$ for the sequence defined by $(y \upharpoonright S)_n:=y_n$ if $n \in S$ and $(y \upharpoonright S)_n:=0$ if $n \notin S$; in addition, $A+B:=\{x+y: x \in A, y \in B\}$ for each $A,B\subseteq \omega$. 
\begin{proof}
[Proof of Theorem \ref{thm:equivalenceweaklynorm}]
The implications \ref{item:1equivalenthyperc} $\implies$ \ref{item:2equivalenthyperc} $\implies$ \ref{item:3equivalenthyperc} are obvious. To prove \ref{item:3equivalenthyperc} $\implies$ \ref{item:1equivalenthyperc}, fix $p \in [1,\infty)$ and let us suppose that $T:=B_{\bm{w}}$ is pointwise $\mathsf{I}$-hypercyclic on $\ell_p$, hence it is possible to pick a sequence $y \in \ell_p$ such that its orbit $\mathrm{orb}(y, T)$ is $\tau^p$-dense 
(the proof in the case where the underlying space is $c_0$ is analogous, hence we omit it). 
Let $(s^{(i)}: i \in \omega)$ be an enumeration of $c_{00} \cap \mathbf{Q}^\omega$. 
For each $i \in \omega$, pick the smallest $m_i\in \omega$ such that $s^{(i)}_n=0$ for all $n> m_i$. Up to relabeling, we can suppose without loss of generality that $m_i \le \max\{1,i\}$ for all $i \in \omega$. For each $i,j \in \omega$, define  
$$
U(i,j):=\left\{x \in \ell_p: |x_n-s^{(i)}_n| < 2^{-j} \text{ for all }n\le m_i\right\},
$$
which is a $\tau^p$-open set. Pick a bijection $h: \omega\to \omega^2$ and, for each $t \in \omega$, define 
$$
S_t:=\left\{n \in \omega: T^ny \in U(h(t))\right\} \in \mathsf{I}^+. 
$$
Since countably generated ideals are clearly $F_\sigma$, it is possible to pick a lscsm $\varphi$ such that $\mathsf{I}=\mathrm{Fin}(\varphi)$ as in \eqref{eq:mazur}. In particular, $\varphi(S_t)=\infty$ for all $t \in \omega$. 

Observe that $\|T\|>1$: indeed, in the opposite, we would have $|w_n|\le 1$ for all $n\ge 1$. Pick $i \in \omega$ such that $s^{(i)}=
(1,0,0,\ldots)$. Thus $S_{h^{-1}(i,1)}=\{n \in \omega: (T^ny)_0 \in (\nicefrac{1}{2},\nicefrac{3}{2})\}\in \mathsf{I}^+$ (in particular, it is infinite). Considering that $(T^ny)_0=w_1\cdots w_n y_n$ for all $n\ge 1$, it follows that there exist infinitely many $n \in\omega$ such that 
$|y_n| \ge |w_1\cdots w_n y_n| = |(T^ny)_0| \ge \nicefrac{1}{2}$. This contradicts the hypothesis that $y \in \ell_p$.\footnote{The same claim $\|T\|>1$ does not hold in larger spaces: in fact, it is well known that the (unilateral unweighted) backward shift $B$ on $\mathbf{R}^\omega$ is pointwise hyperyclic (and has norm $1$), see e.g. the first part of the proof of \cite[Theorem 1]{MR1658096}.}

At this point, set for convenience $F_{-1}:=\{0\}$ and $m_{-1}:=0$, and define recursively a sequence $(F_t: t \in \omega)$ of nonempty finite subsets of $\omega$ as it follows. For each $t \in \omega$, choose a nonempty finite subset $F_t\subseteq S_t$ such that: 
\begin{enumerate}[label={\rm (\roman{*})}]
\item [(a)] $\min(F_t)>\max(F_{t-1})+m_{t-1}$, 
\item [(b)] $\|y \upharpoonright [\min F_t, \infty)\|_p \le 1/2^t\|T\|^{\max F_{t-1}}$ (which is possible since $\lim y=0$), and 
\item [(c)] $\varphi(F_t)\ge t$ (which is possible since $\varphi$ is lower semicontinuous). 
\end{enumerate}

Now, fix $i \in \omega$ and define $F^{(i)}:=\bigcup_{j \in\omega} F_{h^{-1}(i,j)}$.
Considering that 
$$
\varphi(F^{(i)})\ge \sup_{j \in \omega}\varphi(F_{h^{-1}(i,j)})\ge \sup_{j \in \omega} {h^{-1}(i,j)} = \infty,
$$
we obtain that $F^{(i)} \in \mathsf{I}^+$ by the representation \eqref{eq:mazur}. In addition, since countably generated ideals are $Q^+$-ideals, for each $j \in \omega$ there exists $g_{i,j} \in F_{h^{-1}(i,j)}$ such that 
$$
G^{(i)}:=\{g_{i,j}: j \in \omega\}
$$
belongs to $\mathsf{I}^+$. 

To conclude the proof, we claim that 
$$
z:=y \upharpoonright \bigcup\nolimits_{i} (G^{(i)}+[0,m_i]) 
$$
is norm $\mathsf{I}$-hypercyclic (notice that, of course, $z \in \ell_p$).  
In fact, suppose that $n=g_{i,j}$ for some $i,j \in \omega$ and set for simplicity $t:=h^{-1}(i,j)$ so that $g_{i,j} \in F_t$. Then it follows by construction that 
\begin{displaymath}
    \begin{split}
        \|T^{n}z-s^{(i)}\|_p
        &\le \|T^{n}(z \upharpoonright [{n},{n}+m_i])-s^{(i)}\|_p
        +\|T^n (z \upharpoonright (n+m_i,\infty))\|_p\\
        &\le \|T^n(y \upharpoonright [n,n+m_i])-s^{(i)}\|_p
        +\|T^n (y \upharpoonright [\min F_{t+1},\infty))\|_p\\
        &\le 2^{-j}(m_i+1)^{1/p}
        +\|T\|^n \|(y \upharpoonright [\min F_{t+1},\infty))\|_p\\
        &\le 2^{-j}(i+2)^{1/p}+\|T\|^{\max F_t}  \|(y \upharpoonright [\min F_{t+1},\infty))\|_p\\
        &\le 2^{-j}(i+2)+2^{-t}. 
    \end{split}
\end{displaymath}
Taking into account that $\lim_j h(i,j)=\infty$ for each $i \in \omega$, it follows that the subsequence $(T^nz: n \in G^{(i)})$ is norm convergent to $s^{(i)}$. Recalling that $G^{(i)} \in \mathsf{I}^+$, we obtain by \cite[Lemma 3.1(iii)]{MR3920799} that $s^{(i)}$ is an $\mathsf{I}$-cluster point of $\mathrm{orb}(z,T)$. Since the set of $\mathsf{I}$-cluster points is closed, see e.g. \cite[Lemma 3.1(iv)]{MR3920799}, and the set $\{s^{(i)}: i \in \omega\}$ is norm dense in $\ell_p$, we conclude that 
$$
\Gamma_{\mathrm{orb}(z,T)}(\mathsf{I})=\ell_p. 
$$
Therefore $T$ is norm $\mathsf{I}$-hypercyclic. 
\end{proof}

\begin{rmk}\label{rmk:onlycountablygenerated}
    An inspection of the proof of Theorem \ref{thm:equivalenceweaklynorm} reveals that we used the $Q^+$-property and a stronger variant of the $P^+$-property (to obtain $F^{(i)} \in \mathsf{I}^+$). 
    Here, we recall that an ideal $\mathsf{I}$ on $\omega$ is a $P^+$\emph{-ideal} if for all decreasing sequences $(S_n)$ with values in $\mathsf{I}^+$, there exists $S \in \mathsf{I}^+$ such that $S\setminus S_n$ is finite for all $n \in\omega$. It is known that every $G_{\delta\sigma}$-ideal on $\omega$ is a $P^+$-ideal, see \cite[Corollary 2.7]{MSP24} and \cite[Proposition 3.2]{FKL24}. In particular, all $F_\sigma$-ideals (including $\mathrm{Fin}$) are $P^+$-ideals. On the other hand, it is easy to see that $\mathsf{Z}$ is not a $P^+$-ideal, cf. \cite[Figure 1]{FKL24}. 

    This implies that the strategy of the above proof could be adapted, in the best case, for ideals which are both $P^+$ and $Q^+$ (which are commonly known as \emph{selective ideals}). However, it is known that, if $\mathsf{I}$ is an analytic $P$-ideal, then it is selective if and only if it is countably generated. On a different direction, Todor\v{c}evi\'c found an example of an analytic selective
ideal which is not generated by an almost disjoint family; see \cite[Section 1]{MR3351990} and references therein. 
\end{rmk}

\begin{rmk}
   The analogue claim of Theorem \ref{thm:equivalenceweaklynorm} holds replacing $\ell_p$ (or $c_0$) with an arbitrary infinite-dimensional Banach lattice $X\subseteq c_0$ such that $c_{00} \cap \mathbf{Q}^\omega$ is relatively dense in $X$. The proof would go on the same lines, taking into account that norm convergence and pointwise convergence coincide on finite-dimensional subspaces, and that the sequence $z$ belongs to $X$ since $|z|\le |y|$ implies $\|z\|\le \|y\|$. 
\end{rmk}

\medskip

\begin{proof}
    [Proof of Theorem \ref{thm:nonequivalence}]
    Fix $p \in [1,\infty)$, define the function $f: \omega \to \mathbf{R}$ by 
    $$
    f(n):=\left((n+2)\log(n+2)\right)^{1/p},
    $$
    for all $n \in \omega$ and the weight sequence $\bm{w}=(w_0,w_1,w_2,\ldots)$ defined by 
    $
    w_n:=\frac{f(n+1)}{f(n)}
    $ 
    for all $n \in \omega$. 
    Then $\bm{w}$ is decreasing and $\lim_nw_n=1$. 
    We claim that $T:=B_{\bm{w}}$ satisfies the required properties. 

    First, observe that $\sup_n |w_0\cdots w_n|= \sup_n f(n+1)/f(0)=\infty$, hence $T$ is norm $\mathrm{Fin}$-hypercyclic, see e.g. \cite[Theorem 4.8 and Example 4.9]{MR2919812}. In addition, as it follows by Theorem \ref{thm:bayartRuzsa}, $T$ is norm $\mathsf{Z}$-hypercyclic if and only if condition \eqref{eq:conditionufhc} holds. However, the latter one fails since 
    $$
    \sum_{n\in\omega}\frac{1}{(w_0\cdots w_n)^p}\ge \sum_{n\ge 3}\frac{1}{n\log(n)}=\infty.
    $$ 
    Hence $T$ is not norm $\mathsf{Z}$-hypercyclic. To complete the proof, we need only to show that $T$ is pointwise $\mathsf{Z}$-hypercyclic. 
    
    For, let $(s^{(i)}: i \in \omega)$ be an enumeration of $c_{00}\cap \mathbf{Q}^\omega$ and let $m_i \in \omega$ be the smallest \emph{positive} integer such that $s^{(i)}_n=0$ for all $n\ge m_i$. Up to relabeling, suppose without loss of generality that 
    \begin{equation}\label{eq:relabelingsi}
    \forall i \in \omega, \forall n \in \{0,1,\ldots,m_i-1\}, \quad 
    |s^{(i)}_n|\le 2^{i/p}
    \end{equation}
    (notice that this is really possible). Pick also a function $h: \omega \to \omega$ such that $H_k:=h^{-1}(k)$ is infinite for all $k \in \omega$. 
    
    At this point, define recursively an increasing sequence $(n_i: i \in \omega)$ of positive integers as it follows. Set for simplicity $n_{-1}:=0$ and, if $n_{i-1}$ is given for some $i\ge 1$, define 
    \begin{equation}\label{eq:defni}
    n_i:=1+\max\{n_{i-1},2^{h(i)}m_{h(i)}, 3^{2^i(m_{h(i)}+3)^2}\}.
    \end{equation}
    Now, for each $i \in \omega$, define the interval of integers 
    $$
    J_i:=\omega \cap \left[n_i!,\, n_i!\left(1+\frac{n_i}{2^{h(i)}}\right)\right].
    $$
    Define also the finite sequence of rationals $y^{(i)} \in \mathbf{Q}^{<\omega}$ as it follows. Observe that the lenght of $s^{(h(i))}$ is $m_{h(i)}$. Set $q_i:=|J_i|-r_im_{h(i)}$, where $r_i:=\lfloor |J_i|/m_{h(i)}\rfloor$, and define 
    \begin{displaymath}
        \begin{split}
            t^{(i,j)}:=\left(\frac{s^{(h(i))}_0}{w_1w_2\cdots w_{n_i!+jm_{h(i)}}}, 
    \right. &\, \frac{s^{(h(i))}_1}{w_2w_3\cdots w_{n_i!+jm_{h(i)}+1}} 
    , \cdots,  \\
    & \left.
    \frac{s^{(h(i))}_{m_{h(i)}-1}}{w_{m_{h(i)}}w_{m_{h(i)}+1}\cdots w_{n_i!+(j+1)m_{h(i)}-1}} \right)
        \end{split}
    \end{displaymath}
    for each $j \in \{0,1,\ldots,r_i-1\}$. 
    Accordingly, set 
    $$
    z^{(i)}:=t^{(i,0)}
    \text{ }\hspace{-.5mm}^\frown t^{(i,1)}
    \text{ }\hspace{-.5mm}^\frown \ldots
    \text{ }\hspace{-.5mm}^\frown t^{(i,r_i-1)}
    \text{ }\hspace{-.5mm}^\frown 0^{q_i},
    $$
    with the convention that $0^0:=\emptyset$. Notice that $|z^{(i)}|=r_im_{h(i)}+q_i=|J_i|$. 
    Lastly, we claim that the infinite sequence 
    \begin{equation}\label{eq:definizionez}
    z:=z^{(0)} \text{ }\hspace{-.5mm}^\frown z^{(1)} \text{ }\hspace{-.5mm}^\frown z^{(2)} \text{ }\hspace{-.5mm}^\frown 
 \ldots
    \end{equation}
    is pointwise $\mathsf{Z}$-hypercyclic. (Informally, $z$ is a sequence made by the concatenation of all $z^{(i)}$, each of one is supported on $J_i$ and made by suitable rescaled blocks of $s^{h(i)}$.)

    For, we need to show first that the sequence $z$ defined in \eqref{eq:definizionez} belongs to $\ell_p$. Notice that the map $n\mapsto f(n+1)/f(n)$ is decreasing and $n\mapsto f(n)$ is increasing. 
    It follows by \eqref{eq:relabelingsi} and the above construction that 
\begin{displaymath}
        \begin{split}
            \|z\|_p^p
            &=\sum_{i \in \omega}\sum_{n \in J_i}|z^{(i)}_n|^p 
            =\sum_{i \in \omega}\sum_{j\in r_i}\sum_{n \in m_{h(i)}+1}|t^{(i,j)}_n|^p \\
            &\le \sum_{i \in \omega}\sum_{n \in m_{h(i)}}r_i|t^{(i,0)}_n|^p\\
            &\le \sum_{i \in \omega}\sum_{n \in m_{h(i)}}r_i\cdot \frac{\left|s^{(h(i))}_n\right|^p}{(w_{m_{h(i)}}w_{m_{h(i)}+1}\cdots w_{n_i!+m_{h(i)}-1})^p}\\
            &\le \sum_{i \in \omega}\sum_{n \in m_{h(i)}}\frac{|J_i|}{m_{h(i)}}\cdot \frac{2^{h(i)}(w_1w_2\cdots w_{m_{h(i)}-1})^p}{(w_1w_2\cdots w_{n_i!+m_{h(i)}-1})^p}\\
            &= \sum_{i \in \omega}|J_i|\cdot 2^{h(i)} \cdot \left(\frac{f(m_{h(i)})}{f(n_i!+m_{h(i)})}\right)^p\\
            &\le \sum_{i \in \omega}(n_i!\cdot n_i \cdot 2^{-h(i)})\cdot 2^{h(i)}\cdot \frac{(m_{h(i)}+2)\log(m_{h(i)}+2)}{n_i! \log(n_i!)}.
        \end{split}
    \end{displaymath}
Recalling that $n!\ge (n/3)^n$ for all nonzero $n \in \omega$, we obtain by \eqref{eq:defni} that 
\begin{displaymath}
    \begin{split}
        \|z\|_p^p \le &\sum_{i \in \omega}n_i \cdot \frac{(m_{h(i)}+2)\log(m_{h(i)}+2)}{n_i \log(n_i/3)}\\
        &\le \sum_{i \in \omega}\frac{(m_{h(i)}+2)^2}{\log(n_i/3)}\le \sum_{i \in \omega}\frac{1}{2^i}=2.
    \end{split}
\end{displaymath}

To complete the proof, we show that $z$ is pointwise $\mathsf{Z}$-hypercyclic. 
Indeed, for each $j\in \omega$ there exists an infinite set $W_j\subseteq \omega$ such that $h(i)=j$ for all $i \in W_j$. It follows, for each $j \in \omega$, that 
\begin{displaymath}
    \begin{split}
        \mathsf{d}^\star(\{k \in \omega: (T^kz)_n=&s^{(j)}_n \text{ for all }n\in m_j+1\})\ge \limsup_{i \in W_j}\frac{|J_i|/m_j}{\max J_i}\\
        &\ge \limsup_{i \in W_j} \frac{n_i!\cdot n_i/(2^jm_j)}{n_i!(1+n_i/2^j)}=\frac{1}{m_j}>0. 
    \end{split}
\end{displaymath}
Hence, each $s^{(j)}$ is a pointwise $\mathsf{Z}$-cluster point of the orbit of $z$. This completes the proof since the set of $\mathsf{Z}$-cluster points is $\tau^p$-closed and $\{s^{(j)}: j \in \omega\}$ is $\tau^p$-dense in $\ell_p$. 
\end{proof}



\bibliographystyle{amsplain}
\bibliography{idealesz}

\providecommand{\MR}[1]{}
\providecommand{\bysame}{\leavevmode\hbox to3em{\hrulefill}\thinspace}
\providecommand{\MR}{\relax\ifhmode\unskip\space\fi MR }
\providecommand{\MRhref}[2]{%
  \href{http://www.ams.org/mathscinet-getitem?mr=#1}{#2}
}
\providecommand{\href}[2]{#2}
\begin{thebibliography}{10}

\bibitem{MR3543775}
M.~Balcerzak, S.~G\l{a}b, and J.~Swaczyna, \emph{Ideal invariant injections},
  J. Math. Anal. Appl. \textbf{445} (2017), no.~1, 423--442. \MR{3543775}

\bibitem{MSP24}
M.~Balcerzak, S.~G\l{}\c{a}b, and P.~Leonetti, \emph{Topological complexity of
  ideal limit points}, preprint, last updated: Jul 16, 2024
  (\href{https://arxiv.org/abs/2407.12160}{arXiv:2407.12160}).

\bibitem{MR2352487}
F.~Bayart and \'E. Matheron, \emph{Hypercyclic operators failing the
  hypercyclicity criterion on classical {B}anach spaces}, J. Funct. Anal.
  \textbf{250} (2007), no.~2, 426--441. \MR{2352487}

\bibitem{MR2533318}
F.~Bayart and \'{E}. Matheron, \emph{Dynamics of linear operators}, Cambridge
  Tracts in Mathematics, vol. 179, Cambridge University Press, Cambridge, 2009.
  \MR{2533318}

\bibitem{MR3334899}
F.~Bayart and I.~Z. Ruzsa, \emph{Difference sets and frequently hypercyclic
  weighted shifts}, Ergodic Theory Dynam. Systems \textbf{35} (2015), no.~3,
  691--709. \MR{3334899}

\bibitem{MR1658096}
J.~Bonet and A.~Peris, \emph{Hypercyclic operators on non-normable {F}r\'echet
  spaces}, J. Funct. Anal. \textbf{159} (1998), no.~2, 587--595. \MR{1658096}

\bibitem{MR3847081}
A.~Bonilla and K.-G. Grosse-Erdmann, \emph{Upper frequent hypercyclicity and
  related notions}, Rev. Mat. Complut. \textbf{31} (2018), no.~3, 673--711.
  \MR{3847081}

\bibitem{MR2091459}
K.C. Chan and R.~Sanders, \emph{A weakly hypercyclic operator that is not norm
  hypercyclic}, J. Operator Theory \textbf{52} (2004), no.~1, 39--59.
  \MR{2091459}

\bibitem{MR1070713}
J.B. Conway, \emph{A course in functional analysis}, second ed., Graduate Texts
  in Mathematics, vol.~96, Springer-Verlag, New York, 1990. \MR{1070713}

\bibitem{MR4242546}
R.~Ernst, C.~Esser, and Q.~Menet, \emph{{$\mathcal{U}$}-frequent hypercyclicity
  notions and related weighted densities}, Israel J. Math. \textbf{241} (2021),
  no.~2, 817--848. \MR{4242546}

\bibitem{MR4072792}
J.~Falc{\'o} and K.-G. Grosse-Erdmann, \emph{Algebrability of the set of
  hypercyclic vectors for backward shift operators}, Adv. Math. \textbf{366}
  (2020), 107082, 25. \MR{4072792}

\bibitem{MR1711328}
I.~Farah, \emph{Analytic quotients: theory of liftings for quotients over
  analytic ideals on the integers}, Mem. Amer. Math. Soc. \textbf{148} (2000),
  no.~702, xvi+177. \MR{1711328}

\bibitem{FKL24}
R.~Filip\'{o}w, A.~Kwela, and P.~Leonetti, \emph{Borel complexity of sets of
  ideal limit points}, preprint, last updated: Apr 18, 2025
  (\href{https://arxiv.org/abs/2411.10866}{arXiv:2411.10866}).

\bibitem{MR1181163}
J.A. Fridy, \emph{Statistical limit points}, Proc. Amer. Math. Soc.
  \textbf{118} (1993), no.~4, 1187--1192. \MR{1181163}

\bibitem{MR1416085}
J.A. Fridy and C.~Orhan, \emph{Statistical limit superior and limit inferior},
  Proc. Amer. Math. Soc. \textbf{125} (1997), no.~12, 3625--3631. \MR{1416085}

\bibitem{MR1763044}
K.-G. Grosse-Erdmann, \emph{Hypercyclic and chaotic weighted shifts}, Studia
  Math. \textbf{139} (2000), no.~1, 47--68. \MR{1763044}

\bibitem{MR2919812}
K.-G. Grosse-Erdmann and A.~Peris, \emph{Linear chaos}, Universitext, Springer,
  London, 2011. \MR{2919812}

\bibitem{MR2849045}
M.~Hru{\v s}{\'a}k and D.~Meza-Alc{\'a}ntara, \emph{Comparison game on {B}orel
  ideals}, Comment. Math. Univ. Carolin. \textbf{52} (2011), no.~2, 191--204.
  \MR{2849045}

\bibitem{MR1321597}
A.S. Kechris, \emph{Classical descriptive set theory}, Graduate Texts in
  Mathematics, vol. 156, Springer-Verlag, New York, 1995. \MR{1321597}

\bibitem{MR3351990}
A.~Kwela, \emph{A note on a new ideal}, J. Math. Anal. Appl. \textbf{430}
  (2015), no.~2, 932--949. \MR{3351990}

\bibitem{MR3666945}
A.~Kwela and P.~Zakrzewski, \emph{Combinatorics of ideals---selectivity versus
  density}, Comment. Math. Univ. Carolin. \textbf{58} (2017), no.~2, 261--266.
  \MR{3666945}

\bibitem{Leo24}
P.~Leonetti, \emph{Strong universality, recurrence, and analytic {P}-ideals in
  dynamical systems}, preprint, last updated: Jan 02, 2024
  (\href{https://arxiv.org/abs/2401.01131}{arXiv:2401.01131}).

\bibitem{MR3920799}
P.~Leonetti and F.~Maccheroni, \emph{Characterizations of ideal cluster
  points}, Analysis (Berlin) \textbf{39} (2019), no.~1, 19--26. \MR{3920799}

\bibitem{MR4054777}
P.~Leonetti and S.~Tringali, \emph{On the notions of upper and lower density},
  Proc. Edinb. Math. Soc. (2) \textbf{63} (2020), no.~1, 139--167. \MR{4054777}

\bibitem{MR1124539}
K.~Mazur, \emph{{$F_\sigma$}-ideals and {$\omega_1\omega_1^*$}-gaps in the
  {B}oolean algebras {$P(\omega)/I$}}, Fund. Math. \textbf{138} (1991), no.~2,
  103--111. \MR{1124539}

\bibitem{MR4016510}
Q.~Menet, \emph{A bridge between {$\mathcal{U}$}-frequent hypercyclicity and
  frequent hypercyclicity}, J. Math. Anal. Appl. \textbf{482} (2020), no.~2,
  123569, 15. \MR{4016510}

\bibitem{MR1708146}
S.~Solecki, \emph{Analytic ideals and their applications}, Ann. Pure Appl.
  Logic \textbf{99} (1999), no.~1-3, 51--72. \MR{1708146}

\end{thebibliography}
\end{document}